\documentclass[regno]{amsart}
\usepackage{amsmath,amsthm,amsfonts,amssymb,amscd,enumerate,verbatim}
\usepackage{varwidth}
\usepackage{graphicx}
\usepackage[utf8]{inputenc}
\usepackage[colorlinks]{hyperref}
\hypersetup{
 colorlinks=true,
 linkcolor=blue,
 filecolor=magenta, 
 urlcolor=cyan,
}
\usepackage[nameinlink, capitalize]{cleveref}
\usepackage{colonequals}

\usepackage{stmaryrd}

\usepackage[all]{xy}\SelectTips{eu}{}

\usepackage[margin=1.4in]{geometry}

\usepackage{amssymb} 
\usepackage{amsmath} 
\usepackage{amsthm} 
\usepackage{mathtools}

\usepackage{setspace} 
\usepackage{fancyhdr} 
\usepackage{enumerate}  
\usepackage{tikz-cd}

\numberwithin{equation}{section}

\theoremstyle{plain}
\newtheorem{theorem}{Theorem}[section]

\newtheorem*{Theorem1}{Theorem 1}
\newtheorem*{Theorem2}{Theorem 2}
\newtheorem*{Theorem3}{Theorem 3}
\newtheorem*{Main Theorem}{Main Theorem}

\newtheorem{proposition}[theorem]{Proposition}
\newtheorem{lemma}[theorem]{Lemma}
\newtheorem{corollary}[theorem]{Corollary}

{\Alph{theorem}}
{\Alph{theorem}}

\theoremstyle{definition}

\newtheorem{remark}[theorem]{Remark}

  \newcounter{numlist} %
  {\end{list}}%

\theoremstyle{remark}

\newtheorem{chunk}[theorem]{}
\newenvironment{bfchunk}{\begin{chunk}\textbf}{\end{chunk}}

\numberwithin{equation}{theorem}

\numberwithin{equation}{theorem}

\newcommand{\al}{\alpha}
\newcommand{\be}{\beta}
\newcommand{\wt}{\widetilde}

\newcommand{\Z}{\mathbb Z}

\newcommand{\R}{\mathbb R}

\newcommand{\Img}{\mathrm{Im}\,}
\DeclareMathOperator{\gr}{{gr_{\m}}}
\newcommand{\Ker}{\mathrm{Ker}\,}

\newcommand{\ann}{\mathrm{ann}}

\newcommand{\chara}{\mathrm{char}\,}

\newcommand{\pd}{\mathrm{pd}}

\newcommand{\Tor}{\mathrm{Tor}}

\newcommand{\kk}{\sf{k}}

\newcommand{\cx}{\mathrm{cx}}
\newcommand{\ti}{\wt}
\newcommand{\cur}{\mathrm{curv}}
\newcommand{\m}{\mathfrak{m}}
\newcommand{\n}{\mathfrak{n}}
\newcommand{\Ho}{\mathrm{H}}
\DeclareMathOperator{\Po}{P}
\DeclareMathOperator{\Hi}{{\sf H}}

\newcommand\restr[2]{{
  \left.\kern-\nulldelimiterspace 
  #1 
  \vphantom{\big|} 
  \right|_{#2} 
  }}

\begin{document}
\title[The mapping cone of an Eisenbud operator]{The mapping cone of an Eisenbud operator\\ and applications to exact zero divisors} 

\author{Liana M. \c{S}ega}
\author{Deepak Sireeshan}
\begin{abstract}
Let $(Q,\m,\sf{k})$ be a local ring that admits an exact pair of zero divisors $(f,g)$, $M$ a $Q$-module with $fM = 0$ and $U$ a free resolution of $M$ over $Q$.  We construct a degree $-2$ chain map, which we call an Einsenbud operator, on the complex  $U\otimes_QQ/(f,g)$ and use the mapping cone of the operator to study two exact sequences that relate homology over $Q$ to homology over $Q/(f)$. Several applications are given.  
\end{abstract}



\maketitle

\section*{Introduction}
Let $(Q,\m,\kk)$ be a local (meaning also commutative and Noetherian) ring with unique maximal ideal $\m$ and ${\sf k}=Q/\m$. Given a nonzero element $f$ of $Q$, we are concerned with understanding relations between homological behavior over $Q$ and over $R=Q/(f)$.  This problem is well studied when $f$ is a regular element. In this paper we consider $f$ to be a zero divisor. 

If $(F,\partial)$ is a complex of finitely generated $R$-modules, then a classical construction due to Eisenbud \cite{eisenbud} produces degree $-2$ maps $\widetilde{\tau}\colon\widetilde{F}\to \widetilde{F}$, where $(\widetilde F,\widetilde\partial)$ is a lifting of $F$ to $Q$, and thus  $\widetilde F\otimes_QR=F$. With $S=R/\ann_Q(f)R$, the  map 
\[
\tau:=\widetilde\tau\otimes_QS\colon F\otimes_RS\to F\otimes_RS
\]
is  then a degree $-2$ chain map that we refer to as an {\it Eisenbud operator}. When $f$ is a regular element (and thus $S=R$), $\tau\colon F\to F$ is also called a {\it CI operator}, and these maps have been essential in the study of free resolutions and (co)homology over complete intersection rings, see \cite{avramov_ci}. In particular, when $f$ is regular and $M$, $N$ are finitely generated $R$-modules, $\tau$ induces the maps $\gamma_i$ in an exact sequence
\begin{equation}
\label{e:exact-regular}
\cdots \to \Tor_{i+1}^Q(M,N)\xrightarrow{\alpha_{i+1}}\Tor_{i+1}^R(M,N)\xrightarrow{\gamma_i}\Tor_{i-1}^R(M,N)\to \Tor_{i}^Q(M,N)\to \cdots
\end{equation}
in which $\alpha_{i+1}$ is induced by the canonical map $Q\to R$. 
This sequence provides a tool for investigating the change in various homological invariants (e.g.~betti numbers, complexity) via the change of ring $Q\to R$, see for example \cite[Section 3]{Luchonotes}.  

We focus on the case when $f$ is an {\it exact zero divisor}, namely there exists a nonzero element $g\in Q$ such that $\ann_Q(f)=(g)$ and $\ann_Q(g)=(f)$; in this case, the pair $(f,g)$ is called an {\it exact pair of zero divisors}. Such elements bear similarities to non-zero divisors, in that we have some understanding of their homological behavior, due to the simple structure of the minimal free resolutions of the ideals they generate and of their Koszul homology. Exact zero divisors have been studied in recent literature, see \cite{hzerodiv, zerodiv, totref, ekustin}, and the ideals they generate are part of the larger class of quasi-complete intersection ideals studied in \cite{quasi}. Our goal is to explore the Eisenbud operator construction when $f$ is an exact zero divisor, use it to derive results that mirror those obtained for non-zero divisors, and give several applications. 

We start our investigation by observing that, when $f$ is a regular element, the exact sequence \eqref{e:exact-regular} can be understood as coming from the mapping cone $W$ of the Eisenbud operator $\tau$. The homology of $W\otimes_RN$ is isomorphic to $\Tor^Q(M,N)$ in this case, and this observation can be used to show exactness of the sequence \eqref{e:exact-regular}, as seen in \cref{reg_long}. In the case when $f$ is an exact zero divisor and the characteristic of the residue field is zero, we are also able to describe the mapping cone of the Eisenbud operator, and two exact sequences are needed to describe the change of ring, as described below. 
\begin{Theorem1}[\cref{mth}]
Let $(Q,\m,\sf k)$ be a local ring such that $\chara({\sf k}) = 0$. Let $(f,g)$ be an exact pair of zero divisors in $Q$ and set $R = Q/(f)$ and $S= Q/(f,g)$. Let $M$ be an $R$-module and let $(W,\partial^W)$ denote the mapping cone of the Eisenbud operator $\tau$ associated with $f$.

For any $R$-module $N$ with $gN=0$, there are two long exact sequences
\begin{align*}
    \cdots \xrightarrow{}& \Tor_{n-1}^R(M,N) \xrightarrow{}  \Ho_{n}(W \otimes_S N) \xrightarrow{\psi_{n}} \Tor_{n}^R(M,N) \xrightarrow{\delta_{n}} \Tor_{n-2}^R(M,N) \xrightarrow{} \cdots \\ 
    \cdots \xrightarrow{}& \Ho_{n+1}(W \otimes_S N) \xrightarrow{\mu_{n+1}} \Tor_{n-2}^Q(M,N) \xrightarrow{} \Tor^Q_{n}(M,N) \xrightarrow{\phi_n} \Ho_{n}(W \otimes_S N) \xrightarrow{} \cdots 
\end{align*}  
where $\delta^N$ is induced by $\tau$ and  $\psi\phi$ is equal to the map $\Tor^Q(M,N)\to \Tor^R(M,N)$ induced by the canonical projection $Q\to R$. 
\end{Theorem1}
We then apply this theorem to deduce results about vanishing of homology, and relate homological invariants via the change of ring $Q\to R$. Such results have been previously investigated by Bergh, Celikbas, and Jorgensen \cite{hzerodiv} using methods that rely on a change of rings spectral sequence. While we recover many of the results of \cite{hzerodiv} (in a weaker form, because of our additional assumption on characteristic), the assumption that $\chara({\sf k})=0$ leads to new results, and in particular we answer one of the questions raised in \cite{hzerodiv}, as noted after the statement of Theorem 2 below.


In general, for any  $f\in Q$ (not necessarily an exact zero divisor) and $M,N$ finitely generated modules over $R=Q/(f)$ such that $M\otimes_QN$ has finite length, we are concerned with understanding the asymptotic behavior of the size of the homology modules $\Tor_i^Q(M,N)$, and how this behavior changes under the change of ring $Q\to R$. We use $\ell(-)$ to denote length and we define a generalized Poincar\'e series  
$$\Po^Q_{M,N}(t) = \sum_{i=0}^{\infty} \ell(\Tor_i^Q(M,N)) t^i \,.$$
The classical Poincar\'e series of a module $M$, which is the generating series of the sequence of betti numbers of $M$, is $\Po_M^Q(t)=\Po_{M,\kk}^Q(t)$.
We use the invariants {\it length complexity}  and {\it length curvature}, denoted $\ell\cx_Q(M,N)$, respectively $\ell\cur_Q(M,N)$ (see  \cref{s:cx-curv} for the definitions) to measure asymptotic behavior of the coefficients of this series. When $N=\sf k$, these invariants coincide with the more classical definitions of complexity $\cx_Q M$ and curvature $\cur_Q M$ that are  discussed in \cite{Luchonotes}. The behavior of these invariants under the change of ring $Q\to R$ is well understood when $f$ is regular. More precisely, the exact  sequence \eqref{e:exact-regular} gives coefficient-wise inequalities 
\[
\Po^Q_{M,N}(t) \preccurlyeq \Po^R_{M,N}(t)\cdot (1+t)\qquad\text{and}\qquad \Po^R_{M,N}(t)\preccurlyeq \Po^Q_{M,N}(t)\cdot(1-t^2)^{-1}\,
\]
from which one obtains the (in)equalities below, see \cite[Proposition 4.2.5]{Luchonotes}:
\begin{align*}
\begin{split}
\ell\cx_Q(M,N) &\leq \ell\cx_R(M,N) \leq \ell\cx_Q(M,N) + 1 \,, \\
\ell\cur_R(M,N) = &\ell \cur_Q(M,N) \,\, \text{when} \,\,\ell\cur_Q(M,N)\ge 1\,.
\end{split}
\end{align*}

We provide a version of such results in the case when $f$ is an exact zero divisor. We recover two of the results of \cite{hzerodiv}, namely the inequalities providing upper bounds for $\ell\cx_Q(M,N)$ and $\Po^Q_{M,N}(t)$ stated below. We complete the picture with inequalities that provide upper bounds for $\Po^R_{M,N}(t)$ and $\ell\cx_R(M,N)$, and we also address the curvature invariant.   
\begin{Theorem2}[\cref{theorem}, \cref{corol}]
\label{intro2}
Let $(Q,\m, \sf k)$ be a local ring with $\chara({\sf k}) = 0$ and $(f,g)$  an exact pair of zero divisors in $Q$. We set $R = Q/(f)$ and $S = Q/(f,g)$. Let $M$, $N$ be finitely generated $R$-modules such that $\ell(M \otimes_Q N) < \infty$ and $gN=0$. 

The following coefficient-wise inequalities then hold: 
\begin{align}
\label{mainineq}
\Po^R_{M,N}(t) \preccurlyeq \dfrac{(1-t+t^2)\Po^Q_{M,N}(t)}{(1-t)} \quad\text{ and } \quad \Po^Q_{M,N}(t) \preccurlyeq \dfrac{\Po^R_{M,N}(t)}{(1-t)} \,.
\end{align}
Consequently, 
\begin{align*}
\begin{split}
\ell \cx_R(M,N)-1 \leq \ell &\cx_Q(M,N) \leq  \ell \cx_R(M,N) + 1 \qquad\text{and}\\
\ell \cur_Q(M,N) = \ell \cur_R(M,N) \quad &\text{when} \quad  \ell \cur_Q(M,N) \geq 1 \text{ and } \ell \cur_R(M,N) \geq 1 \,.
\end{split}
\end{align*}
\end{Theorem2} 
The authors mention in \cite{hzerodiv} that they do not know whether there exists an $R$-module $M$ with $\cx_Q(M)<\infty$ and $\cx_R(M)=\infty$ (in the context of the theorem, with $N=\sf k$), and ask whether an inequality such as $\cx_R(M)\le \cx_Q(M)$ holds. The inequality $\cx_R(M)-1\le \cx_Q(M)$ resulting from our theorem provides thus a partial answer.  We do not know whether the stronger inequality $\cx_R(M)\le \cx_Q(M)$ holds or if the characteristic assumption can be removed from our result.

While the Poincar\'e series $\Po_M^Q(t)$ are relatively well studied, with several results that imply such series are rational under appropriate assumptions,  there are only a few results (see \cite{gullik}, \cite{melissa}) showing that the generalized Poincar\'e series $\Po_{M,N}^Q(t)$ are rational, or provide formulas for these series. In \cref{s:vanishing} we make further contributions in this direction. We start by exploring the question of when equality holds in the second inequality of \eqref{mainineq}. From \cite{quasi} it is known that equality holds when $f\notin \m^2$ and $\m N=0$. In \cref{vanish} we consider the case when $f\notin \m^2$ and $\m^2 N=0$, and then we prove
the following result, in which $\Hi_M(t)$ denotes the Hilbert series of a $Q$-module  $M$.
\begin{Theorem3}[\cref{final}]
Let $(Q,\m, {\sf k})$ be a local ring with $\m^3=0$. Assume $\chara({\sf k}) = 0$, $(f,g)$ is an exact pair of zero divisors and set $R=Q/(f)$. Let $M,N$ be finitely generated $R$-modules such that $\m(M \otimes_R N) = 0$ and $gN=0$. 
Then 
\begin{equation*}
\label{PMN}
    \Po^Q_{M,N}(t) = \dfrac{\Hi_M(-t)\Hi_N(-t)}{\Hi_Q(-t)} \,.
\end{equation*}
\end{Theorem3}
Local rings $(Q,\m,\kk)$ with $\m^3=0$ that admit exact zero divisors previously appeared in work of Conca \cite{conca} and Avramov, Iyengar and \c{S}ega \cite{avramov}. When $\m^3=0$ and $(f,f)$ is an exact pair of zero divisors, \cref{ez:hil} shows that the element $f$ is a {\it Conca generator}, in the terminology of \cite{avramov}. The result of the theorem can be viewed as an extension of the fact proved in \cite{avramov} that, when $f$ is a Conca generator,  any $Q$-module with $fM=0$ is Koszul. 
 \section{dg algebras and Eisenbud operators}
 Throughout, $(Q,\m, \sf{k})$ denotes a local ring with $\m$ as its unique maximal ideal and ${\sf k}=Q/\m$.
 
In this section, we present some background on dg algebras, and then we extend the classical construction of an Eisenbud operator associated to a regular element $f\in Q$ in order  to allow for $f$ to be a zero divisor. We then give a concrete computation of the operator using the language of dg algebras.

We use $(X,\partial)$, or simply $X$ if $\partial$ is contextually clear,  to denote a sequence of $Q$-modules $X_i$ and $Q$-module homomorphisms $\partial_i$ as follows: 
$$ \cdots \to X_{n+1} \xrightarrow{\partial_{n+1}}  X_{n} \xrightarrow{\partial_{n}}  X_{n-1} \xrightarrow{\partial_{n-1}} \cdots$$
If $\partial_i \partial_{i+1} = 0$ for all $i$, then $(X,\partial)$ is a complex of $Q$-modules. We let $X^\#$ denote the underlying graded $Q$-module $\bigoplus_nX_n$,  and we let $|x|$ denote the degree of a homogeneous element $x \in X$, namely we set $|x| = n$ if $x \in X_n$.  

The {\it shift} of a complex $(X,\partial)$ of degree $i$, denoted $(\Sigma^iX,\Sigma^i\partial)$, is a complex with  
\[
(\Sigma^iX)_n\coloneq X_{n-i}\qquad\text{and}\qquad (\Sigma^i\partial)_n\coloneq (-1)^i\partial_{n-i}\,.
\]

\begin{bfchunk}{Differential graded (dg) algebra/module.} A {\it dg algebra} over $Q$ is a complex $(D,\partial)$ with $D_i = 0$ for all $i < 0$ such that, for $a,b,c \in D$, we have
\begin{itemize}
\item $1 \in D_0$ \text{(Unitary)};
\item $a(bc) = (ab)c$ \text{(Associative)};
\item $ab = (-1)^{|a||b|}ba$ \text{(Graded Commutative)}  \text{and}  $a^2 = 0$  \text{when }$|a|$ \text{ is odd}; 
\item $\partial(ab) = \partial(a)b +(-1)^{|a|}a\partial(b)$  \text{(Leibniz rule)}.   
\end{itemize}
A {\it dg module} $U$ over a dg algebra $D$ is a complex $(U,\partial)$ where $U^\#$ is a $D^\#$-module and for $u \in U$ and $a \in D$, we have
$$\partial(au) = \partial(a)u +(-1)^{|a|}a\partial(u) \, .$$
The dg module $U$ is said to be {\it semi-free} if $U^\#$ is free over $D^\#$. 
\end{bfchunk}

\begin{bfchunk}{Exterior and divided powers variables and Tate resolutions.}\label{sf_ext}
Let $(D, \partial)$ be a dg algebra over $Q$ and let $z$ be a cycle in $D$. Following \cite[Construction 2.1.7, 6.1.1]{Luchonotes}, we describe
a construction that extends  $D$ to a dg algebra $D\langle y\mid \partial(y)=z\rangle$ by adjoining a variable $y$ with $\partial(y) = z$. We also use the shortened notation $D\langle y\rangle$.

\textbf{Case 1: $|z|$ is even.} We define $Q \langle y \rangle$ to be a graded free $Q$-module with basis $\{1,y\}$ such that $|y| = |z| + 1$. We endow  $Q \langle y \rangle$ with a $Q$-algebra structure by setting $y^2 = 0$, and extending by linearity.  The dg algebra $D\langle y\rangle$ is defined as $D \langle y \rangle\coloneqq D \otimes_Q Q \langle y \rangle$, with differential 
$$\partial(d_0 + d_1y) = \partial(d_0) + \partial(d_1)y + (-1)^{|d_1|}d_1z \, .$$
In this case, $y$ is called an exterior variable.  

\textbf{Case 2: $|z|$ is odd.} We define $Q \langle y \rangle$ to be the graded free $Q$-module with basis $y^{(i)}$, where $y^{(0)}=1$, $|y|=|z|+1$ and $|y^{(i)}| = |y|i$ for all $i \geq 0$. We write $y^{(1)}=y$ and we make the convection that $y^{(i)}=0$ when $i<0$.  We endow $Q \langle y \rangle$ with a $Q$-algebra structure by setting
$y^{(i)}y^{(j)} = \binom{i+j}{i} y^{(i+j)} \, ,$
and extending by linearity. The dg algebra $D \langle y \rangle$ is defined as $D \langle y \rangle\coloneqq D \otimes_Q Q \langle y \rangle$, with differential
$$ \partial \left ( \sum_i d_iy^{(i)} \right ) = \sum_i \partial(d_i)y^{(i)} + \sum_i (-1)^{|d_i|}d_i z y^{(i-1)} \, .$$
In this case, $y$ is called a divided powers variable.

The main purpose of presenting this construction is to describe, following \cite[6.3.1]{Luchonotes}, a free resolution over $Q$ of a quotient $R\coloneqq Q/I$ that has a dg-algebra structure (in other words, a dg algebra resolution),  known as a {\it Tate resolution} of $R$ over $Q$. Namely, one successively adjoins finitely many variables in each degree such that the resulting dg algebra, that we shall denote $A$, is a free resolution of $R$ over $Q$. This resolution can be chosen to be minimal when $R=\kk$, but, in general, it may not be minimal, and it requires adjunction of infinitely many variables. We discuss below a situation when the Tate resolution is minimal and requires adjunction of only two variables.  
\end{bfchunk}

\begin{bfchunk}{Exact zero divisors.} 
\label{ezd}Let $f,g\in Q$ nonzero. We say that $(f,g)$ is an {\it exact pair of zero divisors} if $\ann_Q(f) = (g)$ and $\ann_Q(g)= (f)$. An element $f\in Q$ is said to be an {\it exact zero divisor} if there exists $g\in Q$ such that $(f,g)$ is an exact pair of zero divisors.   

Assume $(f,g)$ is an exact pair of zero divisors and set $R=Q/(f)$. Then a Tate resolution $A$ of $R$ over $Q$ is obtained by adjoining one exterior variable $y$ in degree $1$ and one divided power variable $t$ in degree $2$, namely
\begin{equation} \label{A}
A=Q\langle y,t\mid \partial(y)=f, \partial(t)=gy\rangle\, .
\end{equation}
For notation uniformity, we set $y_{2i} = t^{(i)}$ and $y_{2i + 1} = t^{(i)}y$ for all $i$. With this notation, $A$ is a free $Q$-algebra with basis $\{y_i\}_{i\ge 0}$, with $|y_i|=i$ for all $i\ge 0$, its differential is described by 
\[
\partial(y_{2n+1})=fy_{2n} \quad \text{ and } \quad \partial(y_{2n})=gy_{2n-1}
\]
for all $n\ge 0$, and the multiplication is given by 
\[
y_{2i+1}y_{2j+1}=0 \quad \text{ and } \quad  y_{2i}y_{2j+\varepsilon}=\binom{i+j}{i}y_{2(i+j)+\varepsilon}=y_{2j+\varepsilon}y_{2i} 
\]
for all $i\ge 0$, $j\ge 0$, where $\varepsilon\in \{0,1\}$. 

\end{bfchunk}

\begin{bfchunk}{Eisenbud operators.}\label{lift} Let $f\in Q$ and $R=Q/(f)$, and let $(F,\partial)$ be a complex of $R$-free modules. A {\it lifting} of the complex $F$ to $Q$ is a sequence $(\wt{F},\wt{\partial})$ with 
$$
\cdots \to \wt{F}_{n+1}\xrightarrow{\wt{\partial}_{n+1}} \wt{F}_n \xrightarrow{\wt{\partial}_{n}}\wt{F}_{n-1}\to \cdots 
$$
such that $\wt{F}_n$ is a free $Q$-module for all $n$ and 
\[
    \wt{F} \otimes_Q R = F \quad \text{ and } \quad  \wt{\partial} \otimes_Q R= \partial \,.
\]
Note that $(\wt{F},\wt{\partial})$ may not be a complex since $\wt{\partial}^2$ is not necessarily 0. However, since $\partial^2 = 0$, we get that $\wt{\partial}^2 = f\wt{\tau}$ for a map
$$
\wt{\tau}: \Sigma^{-1}\wt{F} \to \Sigma \wt{F}.  
$$  
We set  
$$S\coloneqq\frac{R}{(\ann_Q(f))R}=\frac{Q}{(f)+\ann_Q(f)}
$$
and identify $\wt{F}\otimes_QS=F\otimes_RS$. We define
\begin{equation} \label{tau_def}
\tau\coloneqq \widetilde\tau\otimes_QS\colon \Sigma^{-1}F\otimes_RS\to \Sigma F\otimes_RS \,.
\end{equation}
When $f \in Q$ is a regular element, $\ann_Q(f) = 0$ and hence $S = R$.  In this case, Eisenbud \cite[Proposition 1.1-1.3]{eisenbud} proves that $\tau\colon \Sigma^{-1}F\to \Sigma F$ is a chain map, is independent up to homotopy of the choice of the lifting $(\wt{F},\wt{\partial})$, and is natural, in the sense of \cref{ind_tau}(b) below.  With our extended definition above, we prove in \cref{ind_tau} that these properties hold without the assumption that $f$ is regular.  

Extending the terminology classically used in the case when $f$ is regular, we say that the map $\tau$ in \eqref{tau_def} is {\it the Eisenbud operator} associated to the data $(f, F, \wt{F})$.  A (different) extension of the notion of the Eisenbud operator to the case of exact zero divisors was explored by Windle in \cite{windle}. 
\end{bfchunk}

\begin{lemma} \label{ind_tau}
Let $(Q,\m,\sf{k})$ be a local ring, $f\in Q$ and set $R = Q/(f)$. Let $(F,\partial)$ be a complex of free $R$-modules and let $(\wt{F}, \wt{\partial})$ be a lifting of $F$ to  $Q$.  
If $\tau$ is the Eisenbud operator associated to the data $(f,F,\wt{F})$, then the following hold: 
\begin{enumerate}[{(a)}]
    \item $\tau$ is a chain map;
    \item Let $(F',\partial')$ be another complex of free $R$-modules and $(\wt{F'}, \wt{\partial}')$ be a lifting of $F'$ to $Q$. If $h:F \to F'$ is a chain map of degree 0, and $\tau'$ is the Eisenbud operator associated to the data $(f, F',\wt{F'})$, then $\tau'(h \otimes_Q S)$ is homotopic to $(h \otimes_Q S)\tau$; 
    \item The operator $\tau$ is independent of the choice of the lifting up to homotopy. 
\end{enumerate}
\end{lemma}
\begin{proof}
(a) We need to show $\tau(\partial\otimes_QS)=(\partial\otimes_QS)\tau$, or, in other words 
$$
\left(\wt{\tau}\wt{\partial}- \wt{\partial}\wt{\tau}\right)\otimes_QS= 0.$$
To establish this conclusion, we show $\Img\left(\wt{\tau}\wt{\partial}- \wt{\partial}\wt{\tau}\right)\subseteq \ann_Q(f)\wt{F}$. 

Indeed, since $\wt{\partial}^2 = f\wt{\tau}$, we have
$$f\left(\wt{\tau}\wt{\partial} - \wt{\partial}\wt{\tau}\right) = (f\wt{\tau})\wt{\partial} -\wt{\partial}(f\wt{\tau}) = \wt{\partial}^2\wt{\partial} - \wt{\partial}\wt{\partial}^2 = \wt{\partial}^3-\wt{\partial}^3 = 0 \, .$$

(b) Let $h: F \to F'$ be a chain map. We set $\wt{h}: \wt{F} \to \wt{F'}$ such that $\wt{h} \otimes_Q R = h$ and thus we have 
\begin{equation}\label{theta}
    \wt{\partial'} \wt{h} - \wt{h}\wt{\partial} = f\alpha\, ,
\end{equation} 
for some $\alpha: \wt{F} \to \Sigma\wt{F}'$. To show that  $\tau'(h \otimes_R S)$ is homotopic to $(h \otimes_R S) \tau$, we show 
$$\tau'  (h \otimes_R S) - (h\otimes_R S)  \tau = (\partial' \otimes_R S)(\alpha \otimes_R S) + (\alpha \otimes_R S)  (\partial \otimes_R S)\,.$$
We reach this conclusion by showing
\begin{equation}
\label{ftimes}
f\left(\ti{h}\ti{\tau} - \ti{\tau'}\ti{h} +\ti{\partial'}\alpha + \alpha\ti{\partial}\right)= 0 \, .
\end{equation}
Indeed, by definition of $\tau$ and $\tau'$ we have $f\wt{\tau}=\wt{\partial}^2$ and $f\wt{\tau'}=\wt{\partial'}^2$, hence 
\begin{align*}
    f\ti{h}\ti{\tau} = \ti{h}(f\ti{\tau}) & = \ti{h} \ti{\partial}^2 = (\ti{h} \ti{\partial})\ti{\partial} \\
    & = \left(\ti{\partial'}\ti{h} - f\alpha\right)\ti{\partial} && \text{from $\eqref{theta}$}\\
    & = \ti{\partial'}\ti{h}\ti{\partial} - f\alpha\ti{\partial} \\
    & = \ti{\partial'}\left(\ti{\partial'}\ti{h} - f\alpha\right) - f\alpha\ti{\partial} && \text{from $\eqref{theta}$}\\
    & = \ti{\partial'}^2\ti{h} - \ti{\partial'}f\alpha - f\alpha\ti{\partial} \\
    &= f \ti{\tau'} \ti{h} - f\ti{\partial'}\alpha - f\alpha\ti{\partial}\,.
\end{align*}
This establishes \eqref{ftimes}.

(c) The conclusion follows by taking $F' = F$ and $h = \text{id}_F$ in (b). 
\end{proof}

\begin{bfchunk}{Choice of lifting in the construction of the Eisenbud operator.}\label{desc} Let $f\in Q$ and $R=Q/(f)$, and let $A$ denote a Tate resolution of $R$ over $Q$ obtained by adjoining variables  as in \ref{sf_ext}.  We let $\{y_\lambda\}_{\lambda\in \Lambda} $ denote a basis of $A^\#$ over $Q$, and set $$\Lambda_i\coloneqq\{\lambda\in \Lambda\colon |y_\lambda|=i\}$$   for each $i\ge 0$. Since $(f)$ is principal, we may assume that only one variable in degree $1$ is adjoined and we let $y$ denote the variable with $|y|=1$ and $\partial(y)=f$. Further, we see that \begin{equation}
\label{Lambda2}
\partial(y_\lambda)\in \ann_Q(f)y \qquad\text{for all } y_\lambda\in \Lambda_2\,.
\end{equation}
When $f$ is a non-zero divisor, we take $\Lambda_i=\emptyset$ for $i>1$, and when $f$ is an exact zero divisor, recall from \ref{ezd} that we may take $\Lambda_i=\{i\}$ for all $i\ge 0$. 

Let $M$ be a finite $Q$-module with $fM = 0$. By \cite[Proposition 2.2.7]{Luchonotes}, there exists a free resolution $(U, \partial)$ of $M$ over $Q$ such that $U$ has a semi-free dg module structure over $A$. 
Let $\{e_\delta\}_{\delta \in \Delta}$ denote a basis of $U^\#$ over $A^\#$, and for each $n\in \mathbb Z$ let $V_n$ denote the free $Q$-module with basis $\{e_\delta\ \mid \delta \in \Delta , |e_\delta| = n\}$.  We write then 
\begin{equation}
\label{decomposition}
U_n = V_n\oplus yV_{n-1} \oplus L_n\qquad\text{where }\quad L_n=\bigoplus_{i=2}^n \bigoplus_{\lambda\in \Lambda_i} y_{\lambda} V_{n-i} \, .
\end{equation}
Using the Leibnitz rule and \eqref{Lambda2}, observe:
\begin{equation}
\label{L}
\partial_n(L_n)\subseteq \ann_Q(f)yV_{n-2}+L_{n-1}\,.
\end{equation}
In particular, it follows that $$
(\partial_n\otimes_QS)(L_n\otimes_QS)\subseteq L_{n-1}\otimes_QS\subseteq U_{n-1}\otimes_QS
$$
and we denote  $(L\otimes_QS, \partial\otimes_QS)$ the subcomplex of $U\otimes_QS$ defined in this way. 

 By \cite[Lemma 1.2]{siree},
the complex $(U',\partial')$, with
\begin{equation*}
U_n' = V_n \otimes_Q R \quad\text{ and }\quad 
\partial'_n = \pi_{n-1}\restr{\partial_n}{V_n} \otimes_Q R \, 
\end{equation*}
is a free resolution of $M$ over $R$, where $\pi_{n-1} : U_{n-1} \to V_{n-1}$ is the canonical projection map. We denote by $(V, \partial^V)$ the sequence of free $Q$-modules with $V_n$ as defined above and $$\partial^V_n \coloneqq\pi_{n-1}\restr{\partial_n}{V_n}\colon V_n\to V_{n-1}\,.
$$
Then $(V, \partial^V)$ is a lifting of $U'$ to $Q$. From now on, we will assume that the Eisenbud operator defined in \eqref{tau_def} is associated to the data $(f,U',V)$, and thus we have maps 
\begin{equation*}
    \wt{\tau}\colon \Sigma^{-1}V \to \Sigma V\quad\text{ and } \quad 
    \tau\colon \Sigma^{-1}V \otimes_Q S \to  \Sigma V \otimes_Q S 
\end{equation*}
where $V \otimes_Q S = U' \otimes_R S$. 
\end{bfchunk}
\begin{lemma} \label{taudesc}
    Let $n\in \mathbb Z$ and $x \in V_{n+1}$, and use the decomposition \eqref{decomposition} to write
    $$\partial_{n+1}(x) = x_{n}+yx_{n-1}+\overline{x} \qquad \text{ with } \quad x_i\in V_i,\, \overline{x}\in L_{n}\,.
$$
Then the Eisenbud operator $\tau_n\colon V_{n+1}\otimes_QS\to V_{n-1}\otimes_Q S$ satisfies 
    $$\tau_n(x \otimes 1_S) = -x_{n-1} \otimes 1_S \,.$$   
\end{lemma}

\begin{proof}
Starting with 
\begin{align*}
0=\partial^2(x) = \partial (x_{n}+yx_{n-1}+\overline{x})\,,
\end{align*}
we compute $\partial(x_{n})$ as follows: 
\begin{equation*}
     \partial(x_{n}) =  -fx_{n-1} + y\partial(x_{n-1}) - \partial (\overline{x})\,.\qquad 
     \end{equation*}
Using \eqref{L} for the second equation, observe
$$
-fx_{n-1}\in V_{n-1}\qquad \text{ and }\qquad  y\partial(x_{n-1}) - \partial (\overline{x})\in yV_{n-2}+L_{n-1}\,.   
$$
The definition of $\partial^V$ gives  $\partial^V_{n}(x_{n}) = -fx_{n-1}$. Then, using the definitions of $\wt{\tau}$ and $\tau$,  we have
\begin{equation*}
    \wt{\tau}_n(x) = -x_{n-1} \quad\text{ and }\quad 
    \tau_n(x \otimes 1_S) =-x_{n-1} \otimes 1_S  \,. \qedhere
\end{equation*} 
\end{proof}

\section{The Mapping cone of the Eisenbud Operator} \label{les}
In this section we retain the notation of the previous section and we construct the mapping cone of the Eisenbud operator $\tau\colon \Sigma^{-1}V \otimes_Q S \to \Sigma V \otimes_Q S $ and describe it using the language of dg algebras. When $f$ is a non-zero divisor, we explain that the long exact sequence arising from the mapping cone is exactly the sequence \eqref{e:exact-regular} in the introduction. When $f$ is an exact zero divisor and  $\kk$ has characteristic zero, we describe in \cref{mth} two exact complexes that connect homology over $Q$ to homology over $R$. 

\begin{bfchunk}{The mapping cone of $\tau$.}\label{mc}
Consider the operator $\tau$ as described in \ref{desc}. 
The {\it mapping cone of} $\tau$ is the complex $(W,\partial^W)$ with
\begin{align*}
    W_{n} &= (V_{n-1} \oplus V_{n}) \otimes_Q S = (V_{n-1} \oplus_QS)\oplus( V_{n}\otimes_Q S)\\
     \partial^W_{n} &: (V_{n-1} \oplus V_{n}) \otimes_Q S \to (V_{n-2} \oplus V_{n-1}) \otimes_Q S  \\
    \partial^W_{n} &= \begin{pmatrix}
        -\partial^V_{n-1}  & -\tau_{n-1} \\
        0 & \partial_n^V
    \end{pmatrix}\otimes_QS \,.
\end{align*}
Let $n\in \mathbb Z$,  $x\in V_{n-1}$ and $a\in V_n$. As in  \cref{taudesc},  we write
\begin{align*}
\partial_{n-1}(x) &= x_{n-2} + yx_{n-3} + \overline{x} \\ \partial_n(a) &= a_{n-1} + ya_{n-2} + \overline{a}
\end{align*} 
where $x_{i}, a_i\in V_i$, $\overline{x}\in L_{n-2}$ and $\overline{a}\in L_{n-1}$. Therefore, using \cref{taudesc} for $\tau_{n-1}$, we can describe the differential $\partial^W$ of the mapping cone by $$\partial_{n}^W((x,a) \otimes 1_S)  = (a_{n-2}-x_{n-2}, a_{n-1}) \otimes 1_S \in (V_{n-2} \oplus V_{n-1}) \otimes_Q S.$$

As usual, the mapping cone construction gives rise to a short exact sequence 
\begin{equation}
\label{short}
0\to \Sigma V\otimes_QS\xrightarrow{\zeta} W\xrightarrow{\gamma} V\otimes_QS\to 0 
\end{equation}
of complexes of free $S$-modules.

Let $N$ be an $S$-module. We have an isomorphism of complexes of $S$-modules 
$$(V\otimes_Q S) \otimes_S N \cong (V \otimes_Q R) \otimes_R N =U'\otimes_RN$$
and we identify $\Ho_n(U'\otimes_RN)=\Tor_n^R(M,N). $ Tensoring the short exact sequence above with $N$ yields a short exact sequence that induces the long exact sequence in homology 
\begin{equation}
\label{long}
\dots \to \Tor_{n-1}^R(M,N)\to \Ho_{n}(W\otimes_S N)\to \Tor_{n}^R(M,N)\xrightarrow{\delta_n} \Tor_{n-2}^R(M,N)\to \dots
\end{equation} 
where $\delta_n$ is induced by $\tau_n$. 
\end{bfchunk}

We now proceed to closer investigate the complex $W$. 
\begin{bfchunk}{The map $\omega$.}\label{omega-def} Recall from \eqref{decomposition} that $U_n=V_n\oplus yV_{n-1} + L_n$. For each $n$, we define an $S$-module homomorphism 
 $\omega_n\colon U_n\otimes_Q S\to W_n
 $
 by setting 
 $$ \omega_n(u\otimes 1_S) \coloneqq (x,a)\otimes 1_S ,$$
 where $u\in U_n$ has the following decomposition, per \eqref{decomposition}:
\begin{equation}
\label{u}
u=a+yx+l \qquad\text{ with } a \in V_n\,, x \in V_{n-1}, \text{ and } l \in L_{n} \,.
\end{equation}

\end{bfchunk}

\begin{lemma}\label{omega} 
 The map $\omega\colon U\otimes_Q S\to W$ defined in \ref{omega-def} is a homomorphism of complexes of $S$-modules, and $\Ker(\omega)=L\otimes_QS$. Therefore, 
$$
W\cong \frac{U\otimes_QS}{L\otimes_QS}\,.
$$
If $f$ is a non-zero divisor, then $\omega$ is an isomorphism and $W\cong U\otimes_QS$. 
\end{lemma}
\begin{proof}
Let $n\ge 0$. We decompose $u\in U_n$ as in \eqref{u}. 
Recall that $\omega_n(u\otimes 1_S) =(x,a)\otimes 1_S$. 

To prove that $\omega$ is a chain map, we need to show
$$ \omega_{n-1}(\partial_n \otimes_Q S) - \partial^W_{n} \omega_n = 0$$
for all $n$. Indeed, with the notation introduced in \ref{mc}, we have
$$\partial^W_{n} \omega_n(u \otimes 1_S)= \partial^W_{n}((x,a) \otimes 1_S) = (a_{n-2}-x_{n-2}, a_{n-1}) \otimes 1_S \,.$$
Moreover, we use \eqref{L} to write $l=l_{n-2}+\overline{l}$, where $l_{n-2}\in \ann_Q(f)yV_{n-2}$ and $\overline{l} \in L_{n-1}$, and we use the Leibnitz rule towards the computation below: 
\begin{align*}
\omega_{n-1}&(\partial_n \otimes_Q S)(u \otimes 1_S) \\ & = \omega_{n-1}\left((a_{n-1}+ya_{n-2} +\overline{a} + fx -y(x_{n-2} + yx_{n-3} + \overline{x}) + l_{n-2} + \overline{l}) \otimes 1_S\right) \\
& = \omega_{n-1}\left((a_{n-1} + y(a_{n-2} -x_{n-2}) + \overline{a} -y\overline{x} + \overline{l}) \otimes 1_S\right)\\
&= (a_{n-2}-x_{n-2}, a_{n-1}) \otimes 1_S \,.
\end{align*}
We showed thus $\omega$ is a chain map. It is clearly surjective, and $\Ker(\omega_n) = L_n \otimes_Q S$ for all $n$. This gives the desired isomorphism of complexes. 

When $f \in Q$ is a non-zero divisor, from the discussion in \ref{desc} we have $L_n = 0$ for all $n$, giving that $\omega$ is an isomorphism.
\end{proof}

The following proposition provides an alternate proof for \cite[Proposition 3.2.2]{Luchonotes}, towards explaining the long exact sequence \eqref{e:exact-regular} mentioned in the introduction. 

\begin{proposition} \label{reg_long}
Let $(Q, \m, \sf{k})$ be a local ring, $f \in Q$ a non-zero  divisor and set $R = Q/(f)$. If $M$, $N$ are finitely generated $R$-modules, then there is a long exact sequence 
\begin{align*}
    \cdots \xrightarrow{}& \Tor_{n-1}^R(M,N) \xrightarrow{}  \Tor_{n}^Q(M,N) \xrightarrow{\varphi_n} \Tor_{n}^R(M,N) \xrightarrow{\delta_n} \Tor_{n-2}^R(M,N) \xrightarrow{} \cdots 
\end{align*} 
where $\delta$ is induced by the Eisenbud operator associated with $f$, and $\varphi$ is induced by the canonical projection $Q\to R$. 
\end{proposition}

\begin{proof} Under the given hypothesis on $f$, note that the ring $S$ that we worked with so far is equal to $R$. 
The exactness of the sequence in the statement is a direct consequence of the isomorphism $W\cong U\otimes_QS$ of \cref{omega} and the exact sequence \eqref{long}, using the identification $\Tor_n^Q(M,N)=\Ho_n(U\otimes_QN)$. 

The map $\varphi_n$ in the statement is a composition 
$$
\Ho_n(U\otimes_QN)\to \Ho_n(W\otimes_QN)\to \Ho_n(V\otimes_QN)
$$
where the first map is induced by the isomorphism $\omega$ and the second by the projection $\gamma$ in \eqref{short}, and hence this composition can be viewed as being induced by the canonical projection $\pi\otimes_QR\colon U\otimes_QR\to V\otimes_QR=U'$. Observe that $\pi\otimes_QR$ is a homomorphism of complexes that lifts the identity map on $M$, and thus, upon tensoring with $N$, it describes the canonical map induced in homology by the projection $Q\to R$. 
\end{proof}

We now turn our attention to the situation when $f$ is an exact zero divisor.

\begin{lemma} \label{2sesc}
Let $(Q,\m,\sf k)$ be a local ring with $\chara({\sf k}) = 0$, $(f,g)$ an exact pair of zero divisors, and $S=Q/(f,g)$. With the notation introduced in \ref{mc}, 
\begin{equation}\label{sesq}
    \begin{tikzcd}
       0 \arrow{r} & \Sigma^{2}U \otimes_Q S \arrow[r,"\cdot y_2"] & U \otimes_Q S \arrow[r,"\omega"] & W \arrow{r} & 0
    \end{tikzcd}
\end{equation}
is a short exact sequence of complexes of $S$-modules.  
\end{lemma}
\begin{proof}
Let $\alpha\colon \Sigma^{2}U\to U$ denote the map given by multiplication by $y_2$.

Let $n\ge 0$. 
If $x\in U_{n}$, then
$$
\partial(y_2x)-y_2\partial(x)=gyx+y_2\partial(x)-y_2\partial(x)=gyx
\,.$$
This shows that, upon tensoring with $S$, the map $\alpha$ becomes a chain map. 
To show that $\alpha$ is injective, recall that, when $f$ is an exact zero-divisor, we have
$$
U_{n}=\bigoplus_{i=0}^{n} y_iV_{n-i}
\qquad\text{and}\qquad  y_2 y_{2j+\epsilon} = (1+i)y_{2+2j + \epsilon}\,
$$
for all $n\ge 0$, $j \geq 0$ and $\epsilon \in \{0,1\}$.
When $\chara(\kk)=0$ we see that $1+j$ is a unit for all $i\geq 0$, and thus $y_2y_i=0$ implies $y_i=0$. It follows from here that $\alpha_n\colon U_{n-2}\to U_n$ is injective. Another consequence is that $y_2y_iV_{n-2-i}=y_{i+2}V_{n-2-i}$ and hence
$$\Img(\alpha_n)=y_2\bigoplus_{i=0}^{n-2} y_iV_{n-2-i}=\bigoplus_{i=2}^{n}y_iV_{n-i}=L_n\,. 
$$ 
Since $L_{n}$ is a direct summand of $U_n$, the map $\alpha_n$ remains injective when tensored with $S$.  
Using  \cref{omega}, we further  have
\begin{align*}
\Img(\alpha\otimes_QS)= L\otimes_Q S &= \Ker(\omega). 
\end{align*} 
Since $\omega$ is surjective, this finishes the proof of the exactness of \eqref{sesq}. 
\end{proof}

\begin{theorem} \label{mth}
Let $(Q,\m,\sf k)$ be a local ring such that $\chara({\sf k}) = 0$. Let $(f,g)$ be an exact pair of zero divisors in $Q$ and set $R = Q/(f)$ and $S = Q/(f,g)$. Let $M$ be an $R$-module and let $(W,\partial^W)$ denote the mapping cone of the Eisenbud operator $\tau$ associated with $f$.

For any $R$-module $N$ with $gN=0$, there are have two long exact sequences 
\begin{align*}
    \cdots \xrightarrow{}& \Tor_{n-1}^R(M,N) \xrightarrow{}  \Ho_{n}(W \otimes_S N) \xrightarrow{\psi_{n}^N} \Tor_{n}^R(M,N) \xrightarrow{\delta_{n}^N} \Tor_{n-2}^R(M,N) \xrightarrow{} \cdots \\ 
    \cdots \xrightarrow{}& \Ho_{n+1}(W \otimes_S N) \xrightarrow{\mu_{n+1}^N} \Tor_{n-2}^Q(M,N) \xrightarrow{} \Tor^Q_{n}(M,N) \xrightarrow{\phi_n^N} \Ho_{n}(W \otimes_S N) \xrightarrow{} \cdots 
\end{align*}  
where $\delta^N$ is induced by $\tau$ and $\psi^N\phi^N$ is equal to the map $\Tor^Q(M,N)\to \Tor^R(M,N)$ induced by the canonical projection $Q\to R$. 
\end{theorem}

\begin{proof}
Throughout this proof, we use the notation in \ref{mc}. The first exact sequence has already been established in \eqref{long}. It remains thus to establish the second one. 

Tensoring the short exact sequence \eqref{sesq} of free $S$-modules with $N$, we have the short exact sequence
    \begin{align}\label{nsesr}
        0 \to \Sigma^2U \otimes_Q N \to U \otimes_Q N \xrightarrow{\omega \otimes_R N} W\otimes_QN  \to 0\,.
    \end{align}
 The induced long exact sequence in homology is exactly the second sequence in the statement, with $\Ho_n(U \otimes_Q N) =\Tor_n^Q(M,N)$.  The map $\psi_n\phi_n$ is the composition 
 $$
 \Ho_n(U\otimes_QN)\to \Ho_{n}(W\otimes_SN)\to \Ho_n(V\otimes_QN)
 $$
 where the first map is induced by $\omega$ and the second one is induced by the map $\gamma$ in \eqref{short}. The composition is thus induced by the canonical projection $\pi\otimes_QR\colon U\otimes_QR\to V\otimes_QR=U'$. The homomorphism $\pi\otimes_QR$ lifts the identity map on $M$, and thus, upon tensoring with $N$, it describes the canonical map induced in homology by the projection $Q\to R$. 
\end{proof}

\section{Complexity and Curvature}
\label{s:cx-curv}
In this section, we compare the asymptotic behavior of betti numbers/homology via the canonical projection $Q\to Q/(f)$ where $(Q,\m,\kk)$ is a local ring and $f\in Q$ is an exact zero divisor. We use the notions of complexity and curvature to quantify asymptotic behavior.

\begin{bfchunk}{Poincar\'e series and length complexity.}\label{complex} Suppose $M$, $N$ are finite $Q$-modules with the property that $\ell(M \otimes_Q N) < \infty$ where $\ell$ denotes length. Then, for every nonnegative integer $n$ the length of $\Tor^Q_n(M, N)$ is finite over $Q$. We set $$\be^Q_n(M,N) \coloneqq \ell(\Tor^Q_n(M,N))$$ as defined in \cite[Section 3]{hzerodiv} for example. 

The (generalized) {\it Poincar\'e series } of the pair $(M,N)$ is the formal power series
$$\Po^Q_{M,N}(t) \coloneqq \sum_{i=0}^{\infty} \be_i^Q(M,N)t^i \,.$$

Note that $\Po^Q_{M,\kk}(t)$ is the Poincar\'e series of $M$ over $Q$, usually denoted $\Po^Q_M(t)$. 
The {\it complexity} of a formal power series $p(t) = \sum_{n \ge 0} a_nt^n$ with $a_n \in \R$ for all $n\ge 0$ is defined as follows:
\begin{align*}
\cx(p) &\coloneqq \inf\{c \in \Z^+ \colon \exists \al > 0 \,\,  \text{such that } |a_n| \leq \al n^{c-1} ,\, \forall n \gg 0\}\,.
\end{align*}
In this definition, observe that we can replace $n\gg 0$ with $\forall n$. 
Indeed, when $\cx(p) = c$, we have $|a_N| \le \al_1N^{c-1}$ for some $N \gg 0$. We can set $\al = \max\{|a_i|, \al_1\}$ for $i \le N-1$ which guarantees the inequality for all $n$. 

The {\it length complexity of the pair} $(M,N)$, denoted by $\ell\cx_Q(M,N)$, is defined by
$$ \ell\cx_Q(M,N) \coloneqq \cx(\Po^Q_{M,N}(t))\,.$$ 
The {\it complexity} of a $Q$-module $M$ is $\cx_Q(M)\coloneqq \ell\cx_Q(M,{\sf k})$.

\end{bfchunk}
\begin{lemma}\label{comp}
If $p$ and $q$ are formal power series with real coefficients, then 
$$\cx(p \cdot q) \leq \cx(p) + \cx(q) \,.$$
\end{lemma}

\begin{proof}
Let $p(t) = \sum_{i=0}^{\infty} a_nt^n, q(t) = \sum_{i = 0}^
{\infty} b_nt^n$, $\cx(p) = c$ and $\cx(q) = c'$. Denote the $n$th-coefficient of the product of the power series $p$ and $q$ by $d_n$. Notice that $d_n = \sum_{i=0}^na_ib_{n-i}$. 

For big enough $\al$ and $\al'$ and for all $n$ we have 
\begin{align*}
    \left |a_ib_{n-i} \right | &\leq \al(i^{c-1}(n-i)^{c'-1}) \leq \al(n^{c-1}(n-i)^{c'-1}) 
    \leq \al' n^{c+c'-2} \\
     |d_n|&=\left | \sum_{i=0}^na_ib_{n-i} \right |\le \sum_{i=0}^n \left |a_ib_{n-i} \right |
      \leq n \al'n^{c+c'-2} 
     = \al'n^{c+c'-1} \, .
\end{align*}
By the definition of complexity, we have the desired conclusion. \end{proof}

If $p(t) = \sum_{n \ge 0} a_nt^n, q(t) = \sum_{n \ge 0} b_nt^n$ are formal power series with real coefficients, coefficient-wise inequality is denoted $\preccurlyeq$ and is defined as follows:
$$p(t) \preccurlyeq q(t) \,\, \iff\,\, a_n \leq b_n \,\, \text{for all} \, n \,.$$

\begin{theorem}\label{theorem}
Let $(Q,\m, \sf k)$ be a local ring with $\chara({\sf k}) = 0$ and $(f,g)$ is an exact pair of zero divisors in $Q$. We set $R = Q/(f)$ and $S = Q/(f,g)$. Then, for every finitely generated $R$-modules $M$, $N$ with $gN = 0$ and $\ell(M \otimes_Q N) < \infty$ we have 
\begin{align}\label{poincare1}
\Po^R_{M,N}(t) &\preccurlyeq \dfrac{(1-t+t^2)\Po^Q_{M,N}(t)}{(1-t)} \qquad \text{and}\\
\label{poincare2}
\Po^Q_{M,N}(t) &\preccurlyeq \dfrac{\Po^R_{M,N}(t)}{(1-t)} \,.
\end{align}
Furthermore, equality holds in \eqref{poincare2} if and only if the maps $\mu^N$ and $\delta^N$ from the exact sequences of \cref{mth} are both equal to zero. 
\end{theorem}

\begin{proof}
Recall the following long exact sequence from \cref{mth}. 
\begin{equation} \label{lesq}
\cdots\xrightarrow{\mu_{n+3}^N} \Tor_n^Q(M,N) \to \Tor^Q_{n+2}(M,N) \to \Ho_{n+2}(W \otimes_S N) \xrightarrow{\mu_{n+2}^N} \Tor^Q_{n-1}(M,N) \to \cdots 
\end{equation}
A length count gives
$$
\ell(\Ho_{n+2}(W \otimes_S N)) \leq \be_{n+2}^Q(M,N) + \be_{n-1}^Q(M,N)
$$
for all $n\ge 0$. Using these inequalities for all $n$, we further obtain
\begin{align}\label{poiq}
\begin{split}
    \sum_{n=-2}^\infty \ell(\Ho_{n+2}(W\otimes_S N))t^{n+2} &\preccurlyeq \sum_{n=-2}^\infty \be_{n+2}^Q(M,N)t^{n+2} + \sum_{n=-2}^\infty \be_{n-1}^Q(M,N) t^{n+2}\\
    & = (1 + t^3) \Po^Q_{M,N}(t)\,.
\end{split}
\end{align}

Similarly, we use the other long exact sequence from \cref{mth}, namely 
\begin{equation}\label{lesr}
\cdots \xrightarrow{\delta_{n+3}^N} \Tor_{n+1}^R(M,N) \to \Ho_{n+2}(W\otimes_S N) \to \Tor_{n+2}^R(M,N)\xrightarrow{\delta_{n+2}^N} \Tor_{n}^R(M,N)\to \cdots
\end{equation}

A length count gives
$$
  \ell(\Ho_{n+2}(W\otimes_S N)) \geq \be_{n+2}^R(M,N) - \be_{n}^R(M,N)\,.
$$
Using these inequalities for all $n$, we have:
\begin{align}\label{poir}
\begin{split}
    \sum_{n=-2}^\infty \ell(\Ho_{n+2}(W\otimes_S N))t^{n+2} &\succcurlyeq \sum_{n=-2}^\infty \be_{n+2}^R(M,N)t^{n+2} - \sum_{n=-2}^\infty \be_{n}^R(M,N) t^{n+2}\\
    & = (1 - t^2) \Po^R_{M,N}(t)\,.
\end{split}
\end{align}
From the results of \eqref{poiq} and \eqref{poir}, we get
$$ (1-t) \Po^R_{M,N}(t) \preccurlyeq (1-t + t^2)\Po^Q_{M,N}(t) \,.$$
Since $0\preccurlyeq(1-t)^{-1}$, the inequality is preserved by multiplying both sides by $(1-t)^{-1}$ and this gives the first inequality of the theorem. 

To get the second inequality, we use again a length count in \eqref{lesq}:
$$
\ell(\Ho_{n+2}(W \otimes_S N)) \geq \be_{n+2}^Q(M,N) - \be_{n}^Q(M,N)\
$$
with equality if and only if $\mu_{n+2}^N=0=\mu^N_{n+3}$. Using these inequalities for all $n$, we further obtain
\begin{align}\label{poiq2}
\begin{split}
    \sum_{n=-2}^\infty \ell(\Ho_{n+2}(W\otimes_S N))t^{n+2} &\succcurlyeq \sum_{n=-2}^\infty \be_{n+2}^Q(M,N)t^{n+2} - \sum_{n=-2}^\infty \be_{n}^Q(M,N) t^{n+2}\\
    & = (1 - t^2) \Po^Q_{M,N}(t) \,,
\end{split}
\end{align}
with equality if and only if $\mu^N=0$. 
Similarly, a length count in \eqref{lesr} gives
$$
\ell(\Ho_{n+2}(W\otimes_S N)) \leq \be_{n+2}^R(M,N) + \be_{n+1}^R(M,N)
$$
with equality if and only if $\delta_{n+2}^N=0=\delta_{n+3}^N$. Using these inequalities for all $n$ we further have
\begin{align}\label{poir2}
\begin{split}
    \sum_{n=-2}^\infty \ell(\Ho_{n+2}(W\otimes_S N))t^{n+2} &\preccurlyeq \sum_{n=-2}^\infty \be_{n+2}^R(M,N)t^{n+2} + \sum_{n=-2}^\infty \be_{n+1}^R(M,N) t^{n+2}\\
    & = (1 + t) \Po^R_{M,N}(t) \,,
\end{split}
\end{align}
with equality if and only if $\delta^N=0$. 
From \eqref{poiq2} and \eqref{poir2}, we get
$$(1-t^2)\Po^Q_{M,N}(t) \preccurlyeq (1+t)\Po^R_{M,N}(t) \,.$$ 
Observing that $0\preccurlyeq (1-t^2)^{-1}$, we then obtain the desired inequality by multiplying both sides with $(1-t^2)^{-1}$.  

\end{proof}

\begin{corollary} \label{connec}
With the hypotheses of \cref{mth}, if $f\in \m\smallsetminus \m^2$ and $\m N=0$, then $\mu^{N}=0=\delta^{N}$. 
\end{corollary}
\begin{proof}
When $f\notin \m^2$, it follows from \cite[Theorem 6.2]{quasi} that we have an equality
$$\Po^Q_{M,{\sf k}}(t) = {\Po^R_{M,{\sf k}}(t)}/{(1-t)}\,.
$$ 
Obviously, if $\m N=0$, then we can replace $\kk$ with $N$ in this equality. Consequently, equality holds in $\eqref{poincare2}$, and \cref{theorem} gives the desired conclusion. 
\end{proof}
\begin{remark}
The inequality $\Po^Q_{M,N}(t) \preccurlyeq {\Po^R_{M,N}(t)}/{(1-t)}$ was previously proved using spectral sequences, without any characteristic assumption, in \cite[Theorem 3.2]{hzerodiv}.  

\end{remark}

\begin{corollary} \label{corol}
Let $(Q,\m, \sf k)$ be a local ring with $\chara({\sf k}) = 0$, $(f,g)$ an exact pair of zero divisors in $Q$ and set $R=Q/(f)$. For every finitely generated $R$-modules $M$, $N$ with $gN = 0$ and $\ell(M \otimes_Q N) < \infty$, we have 
     $$\ell \cx_Q(M,N) - 1 \leq \ell \cx_R(M,N)  \leq  \ell \cx_Q(M,N) + 1\,.$$
\end{corollary}
\begin{proof}

Using \cref{comp}, \eqref{poincare1} and \eqref{poincare2}, we have the following inequalities of length complexities:
\begin{align*}
\ell\cx_R(M,N)  &\le \ell\cx_Q(M,N)  + \cx((1-t)^{-1}) + \cx(1 - t + t^2) \\
& = \ell\cx_Q(M,N) + 1 \\
\ell\cx_Q(M,N)  &\le \ell\cx_R(M,N)  +  \cx((1-t)^{-1}) \\
 & = \ell\cx_R(M,N) + 1 \,.\qedhere
\end{align*}
\end{proof}
\begin{remark}
When $f \in Q$ is a regular element, upper and lower bounds for $\cx_R(M)$ in terms of $\cx_Q(M)$ can be derived using \eqref{e:exact-regular} and are spelled out in the introduction.  Consequently, we have $\cx_R(M) < \infty \iff \cx_Q(M) < \infty$ in this case. When $f$ is an exact zero divisor, Bergh, Celikbas and Jorgensen \cite[Corollary 3.3]{hzerodiv} provide the upper bound in \cref{corol} for $\ell\cx_Q(M,N)$ (without any assumption on the characteristic) and question whether a lower bound for $\ell\cx_Q(M,{\sf k}) = \cx_Q(M)$ can be given using $\cx_R(M)$. Thus \cref{corol} provides a partial answer to this question. 
\end{remark}

\begin{bfchunk}{Curvature.}\label{curvature} Let $p(t) = \sum_n a_n t^n$ be a formal power series with real coefficients and radius of convergence $r(p)$. We define the {\it curvature} of $p(t)$ by
$$
\cur(p)\coloneq\frac{1}{r(p)}=\limsup_{n \to \infty} \sqrt[n]{|a_n|}\,.
$$
If $M$ is a finite $Q$-module then the curvature of the module $M$, denoted  $\cur_Q(M)$, is defined as the curvature of $\Po_M^Q(t)$, see \cite{Luchonotes}. 
When $\ell(M\otimes_QN)<\infty$, we define the length curvature of the pair $(M,N)$ to be the curvature of $\Po^Q_{M,N}(t)$. 

Notice that for any two power series $p$ and $q$ with real coefficients, we have
$$ \cur(p\cdot q) \leq \max\{\cur(p),\cur(q)\} \,.$$ 
This is a consequence of the fact that the radius of convergence of the product $p\cdot q$ is equal to  the minimum of the radii of convergence of $p$, $q$. 
\end{bfchunk}
 
\begin{proposition}
Let $R = Q/(f)$ where $(Q,\m, \sf k)$ is a  local ring with $\chara({\sf k}) = 0$ and $(f,g)$ is an exact pair of zero divisors in $Q$. Let $M$, $N$ be finitely generated $R$-modules $M$, $N$ such that $gN = 0$ and $\ell(M \otimes_Q N) < \infty$. 

If $\ell\cur_Q(M,N) \geq 1$ and $\ell\cur_R(M,N)\geq 1$ then 
\begin{align}
\ell \cur_Q(M,N) = \ell \cur_R(M,N)\,.
\end{align}
\end{proposition}
\begin{proof}
From \eqref{poincare1}, \eqref{poincare2} and the hypothesis $\ell\cur_Q(M,N) \geq 1$ and $\ell\cur_R(M,N)\geq 1$, we get the following inequalities: 
\begin{align*}
    \ell \cur_R(M,N) &\leq \max \{\ell \cur_Q(M,N), \cur(1-t+t^2), \cur((1-t)^{-1}) \} \\
    & = \max \{\ell \cur_Q(M,N), 1 \}\\
    &=\ell\cur_Q(M,N)\\ 
    \ell \cur_Q(M,N) &\leq \max \{\ell \cur_R(M,N) , \cur((1-t)^{-1})\} \\
    &=\max \{\ell \cur_R(M,N), 1\}\\
    &=\ell\cur_R(M,N)\,. \qedhere
\end{align*} 
\end{proof}

\begin{remark}
When $f \in Q$ is a regular element, we know $\cur_Q(M) = \cur_R(M)$ when $\pd_Q(M)=\infty$, see \cite[Proposition 4.2.5]{Luchonotes}. Our result implies similar behavior when $f$ is an exact zero divisor, namely $\cur_Q(M) = \cur_R(M)$ when $\pd_Q(M)=\infty$ and $\pd_R(M)=\infty$
\end{remark}
\section{Vanishing of homology and other applications}
\label{s:vanishing}
In this section we present more applications of \cref{mth}. The setting is, as before, that $(Q,\m, \kk)$ is a local ring, $(f,g)$ is an exact pair of zero divisors in $Q$, $R=Q/(f)$ and $M, N$ are finitely generated $R$-modules with $gN=0$. The first result relates vanishing of homology over $R$ to vanishing of homology over $Q$. Then, we study the maps induced in homology by the inclusion $\n N\subseteq N$, and we use this exploration to give in \cref{final} a formula for $\Po_{M,N}^Q(t)$ in a special case.

The next proposition recovers the result of \cite[Theorem 2.1]{hzerodiv} and extends it, by requiring a smaller range of vanishing in the hypothesis.

\begin{proposition} \label{vh}
Let $(Q,\m,\kk)$ be a local ring with $\chara({\sf k})=0$, $(f,g)$ an exact pair of zero divisors, and set $R = Q/(f)$. Let $M$, $N$ be $R$-modules such that $gN = 0$. If there exists integers $n, m\geq 0$ with $n - m\ge 1$ and $$
\Tor_i^R(M,N) = 0 \qquad\text{for  }\quad  m \leq i \leq n
$$
then 
$$\Tor_{i-1}^Q(M,N) \cong \Tor_{i+1}^Q(M,N)
\qquad\text{for}\quad  m \leq i \leq n - 2
$$
and there is an exact sequence:
$$ \Tor^Q_{m-2}(M,N) \xrightarrow{}  \Tor^Q_{m}(M,N) \xrightarrow{} \Tor^R_{m-1}(M,N) \xrightarrow{} \Tor^Q_{m-3}(M,N).$$
Furthermore, if $m=1$, then 
$$M \otimes_Q N \cong \Tor_i^Q(M,N) \qquad \text{for} \quad 1 \leq i \leq n - 1.$$
\end{proposition}
\begin{proof}
Under the assumptions in the statement, \cref{mth} gives a long exact sequence
$$\Tor_{i+2}^R(M,N) \to \Tor_{i}^R(M,N) \to \Ho_{i+1}(W \otimes_S N) \to \Tor_{i+1}^R(M,N) \to \Tor_{i-1}^R(M,N) $$ from which we obtain: 
\begin{equation} \label{vh-cong}
    \Ho_{i+1}(W \otimes_S N) \cong \begin{cases}
    0 & \text{for } m \leq i \leq n-1\\
    \Tor^R_{i+1}(M,N) & \text{for } i = n \\
    \Tor^R_{i}(M,N) & \text{for } i = m - 1
    \end{cases}
\end{equation}
The second exact sequence of \cref{mth} is  
$$
\cdots \xrightarrow{} \Ho_{i+1}(W \otimes_S N) \xrightarrow{\mu_{i+1}^N} \Tor_{i-2}^Q(M,N) \xrightarrow{} \Tor^Q_{i}(M,N) \xrightarrow{\phi_n^N} \Ho_{i}(W \otimes_S N) \xrightarrow{} \cdots 
$$
The exact sequence in the statement and the isomorphisms $\Tor_{i-1}^Q(M,N) \cong \Tor_{i+1}^Q(M,N)$ for $m\leq i\leq n$ follow by using \eqref{vh-cong} into this sequence. 

When $m=1$, we have thus
\begin{equation*}
    \Tor_i^Q(M,N) \cong \begin{cases}
    M \otimes_Q N & \text{when $i$ is even}\\
    \Tor_1^Q(M,N) & \text{when $i$ is odd}\\
    \end{cases}
\end{equation*}
for  $1 \leq i \leq n-1$. 
Using again \eqref{vh-cong} in the second sequence from \cref{mth}, we conclude
\begin{align*}
&\Ho_{0}(W \otimes_S N) \cong M \otimes_Q N \cong \Ho_{1}(W \otimes_S N) \cong \Tor^Q_{1}(M, N) \,. \qedhere
\end{align*}
\end{proof} 

\begin{bfchunk}{Induced maps.} 
Let  $M, N, P$ be $ Q$-modules,  $\phi: M \to N$  a homomorphism and  $U$ a complex of $Q$-modules. We let then $U \otimes_Q \phi$ and 
$\Tor_i^Q(P, \phi)$ denote the induced maps
\begin{align*}
U \otimes_Q \phi &: U \otimes_Q M \to U \otimes_Q N \\
\Tor_i^Q(P, \phi) &: \Tor_i^Q(P, M) \to \Tor_i^Q(P, N) \,.
\end{align*}
Further, let 
$$\nu_N: \m N \to N  \quad \text{and} \quad \pi_N: N \to \overline{N} = N/\m N$$
denote the inclusion and the projection map respectively. 
\end{bfchunk}
As mentioned in the proof of \cref{connec}, equality holds in \eqref{poincare2} when $\m N=0$ and $f\notin \m^2$.  
In the next result, we identify another situation in which equality holds in \eqref{poincare2}.
\begin{proposition} \label{vanish}
Let $(Q, \m, \kk)$ be a local ring with $\chara({\sf k}) = 0$ and $(f,g)$  an exact pair of zero divisors with $f\in \m\smallsetminus \m^2$. Set $R = Q/(f)$ and let $M$, $N$ be $R$-modules such that $\m^2N = 0 = gN$. 

If $\Tor_i^R(M,\nu_N) = 0$ for all $i\ge 0$  then $\Tor_i^Q(M,\nu_N) = 0$ for all $i\ge 0$ and 
$$\Po^Q_{M,N}(t)=\frac{\Po^R_{M,N}(t)}{1-t}\,.$$
\end{proposition}
\begin{proof} 
Set $S=R/(g)$. Consider the  short exact sequence of $R$-modules:
\begin{equation}\label{sesm}
0 \to \m N \xrightarrow{\nu_N} N \to \overline{N} \to 0
\end{equation} where $\overline{N} = N/\m N$.  This sequence remains exact when tensoring with the complexes of free $S$-modules $\Sigma V\otimes_QS$, $W$, respectively $V\otimes_QS$. We also consider the sequence  of complexes of  free $S$-modules \eqref{short} and we tensor with $\m N$, $N$, respectively $\overline N$; note the resulting sequences also remain exact. We obtain thus a commutative diagram of complexes with exact rows and columns, where the columns are induced from \eqref{sesm} and the rows are induced by \eqref{short}.  We omit writing this diagram, and we write directly the following diagram induced in homology, cf.~\cite[Chapter IV, Proposition 2.1]{cartan}.   

\begin{equation*}
    \begin{tikzcd}[column sep=5mm, row sep=5mm]
     \cdots \arrow[r, "\delta_{n+2}^N"] & \Tor_{n}^R(M,{N}) \arrow{r} \arrow[d, hook] &  \Ho_{n+1}(W \otimes_S {N}) \arrow[r] \arrow[d]&  \Tor_{n+1}^R(M,{N}) \arrow[r, "\delta^{N}_{n+1}"] \arrow[d, hook]& \cdots\\
        \cdots \arrow[r,"\delta^{\overline{N}}_{n+2}","0"'] & \Tor_{n}^R(M,\overline{N}) \arrow[r, hookrightarrow] \arrow[d, twoheadrightarrow] & \arrow[d] \Ho_{n+1}(W \otimes_S \overline{N}) \arrow[r,twoheadrightarrow] & \arrow[d, twoheadrightarrow] \Tor_{n+1}^R(M,\overline{N}) \arrow[r,"\delta^{\overline{N}}_{n+1}","0"'] & \cdots\\
       \cdots \arrow[r,"\delta^{\m N}_{n+1}","0"']  & \Tor_{n-1}^R(M,{\m N}) \arrow[r, hookrightarrow] \arrow[d,"0"',"{\Tor_{n-1}^R(M,\nu_N)}"] & \arrow[d, "{\Ho_{n}(W \otimes_S \nu_N)}"] \Ho_{n}(W \otimes_S {\m N}) \arrow[r, twoheadrightarrow] & \arrow[d,"0"',"{\Tor_{n}^R(M,\nu_N)}"] \Tor_{n}^R(M,{\m N}) \arrow[r,"\delta^{\m N}_{n}","0"'] & \cdots \\
      \cdots \arrow[r] & \Tor_{n-1}^R(M,{N}) \arrow{r}  &  \Ho_{n}(W \otimes_S {N}) \arrow{r} &  \Tor_{n}^R(M,{N}) \arrow[r] & \cdots
    \end{tikzcd}  
\end{equation*}
In this diagram, all rows and columns are exact.  \cref{connec} shows that the connecting homomorphisms $\delta^{\overline N}$ and $\delta^{\m N}$ are zero. The maps $\Tor^R_n(M,\nu_N)$ are zero for all $n$, by hypothesis. Thus, some of the maps in the diagram are injective, respectively surjective, as indicated. Applying Snake lemma to the middle two rows of the commutative diagram above for all values of $n$, we get
\begin{equation} \label{res}
\Ho_{n}(W \otimes_S \nu_N) = 0  \qquad \text{and}\qquad \partial_{n+1}^N=0\qquad \text{for all } n \,.
\end{equation}

We now use the short exact sequences \eqref{sesq} and \eqref{sesm} to similarly get the following commutative diagram with exact rows and columns.
\begin{equation*}
        \begin{tikzcd}[column sep=5mm, row sep=5mm]   
           {} \arrow[r,"\mu_{n+3}^N"] &\Tor_n^Q(M,N) \arrow[d] \arrow[r] & \Tor_{n+2}^Q(M,{ N}) \arrow[r] \arrow[d] & \arrow[d,hook] \Ho_{n+2}(W \otimes_S {N}) \arrow[r,"\mu^{N}_{n+2}"]& {} \\
           {} \arrow[r,"0"',"\mu_{n+3}^{\overline{N}}"] &\Tor_n^Q(M,\overline{N}) \arrow[d] \arrow[r, hookrightarrow] & \Tor_{n+2}^Q(M,\overline{N}) \arrow[r, twoheadrightarrow] \arrow[d] &  \Ho_{n+2}(W \otimes_S \overline{N}) \arrow[d, twoheadrightarrow] \arrow[r,"0"',"\mu_{n+2}^{\overline{N}}"] & {}\\
            {} \arrow[r, "{\mu_{n+2}^{\m N}}", "0"'] & \Tor_{n-1}^Q(M,\m N) \arrow[r, hookrightarrow] \arrow[d, "{\Tor_{n-1}^Q(M,\nu_N)}"'] & \Tor_{n+1}^Q(M,{\m N}) \arrow[r, twoheadrightarrow]\arrow[d, "{\Tor_{n+1}^Q(M,\nu_N)}"'] &  \Ho_{n+1}(W \otimes_S {\m N}) \arrow[r, "\mu_{n+1}^{\m N}", "0"']\arrow[d,"{\Ho_{n+1}(W\otimes_S\nu_N)}"',"0"]&{}\\
            {} \arrow{r} &\Tor_{n-1}^Q(M,N) \arrow[r] & \Tor_{n+1}^Q(M,{ N}) \arrow[r]  &  \Ho_{n+1}(W \otimes_S {N}) \arrow[r,"\mu^{N}_{n+1}"]& {} 
        \end{tikzcd}  
\end{equation*}
 \cref{connec} shows that the connecting homomorphisms $\mu^{\overline N}$ and $\mu^{\m N}$ are zero, and \eqref{res} gives $\Ho_{n+1}(W \otimes_S \nu_N) = 0$.
Thus, some of the maps in the diagram are injective, respectively surjective, as indicated.

We prove by induction on $n$ that $\Tor_n^Q(M,\nu_N)=0$ for all $n$. The statement is obviously true when $n=-1$ and $n=-2$. Let $n\ge -1$ and assume that $\Tor_i^Q(M,\nu_N)=0$ for all $i\le n$, and in particular for $i=n-1$. A use of the Snake Lemma for the middle two rows of the diagram gives that $\Tor_{n+1}^Q(M,\nu_N)=0$. This finishes the induction. Another use of the Snake Lemma for the same rows also yields $\mu^N_{n+2}=0$ for all $n$.

Then \cref{theorem} shows that equality must hold in \eqref{poincare2}. 
\end{proof}

\begin{bfchunk}{Hilbert Series and Koszul modules.}\label{koszul:def} Let $(Q,\m, \sf k)$ be a local ring. We define the Hilbert series of $M$ over $Q$ as
$$\Hi_M(t) = \sum_{n = 0}^{\infty} \dim_{\sf k}(\m^{n}M/\m^{n+1}M)t^n \,.$$

A $Q$-module $M$ is said to be {\it Koszul} if the associated graded module $\gr(M)$ over the associated graded ring $\gr(Q)$ has a linear resolution. When $M$ is Koszul, \cite[1.8]{herzog} gives the formula
\begin{equation}\label{koszul:p}
\Po^Q_M(t)=\frac{\Hi_M(-t)}{\Hi_R(-t)} \,.
\end{equation}
\end{bfchunk}

A similar formula is not known for generalized Poincar\'e series $\Po^Q_{M,N}(t)$ when both modules $M$, $N$ are Koszul.

One of the few situations when it is possible to find expressions for $\Po^Q_{M,N}(t)$ is when the square of the maximal ideal is zero, which is the hypothesis in the lemma below. This result will allow us to verify some of the hypotheses in \cref{final}, in order to further compute generalized Poincar\'e series over rings with the cube of the maximal ideal equal to $0$. 

\begin{lemma}
\label{n2=0}
If $(R,\n,\kk)$ is a local ring with $\n^2=0$, and $M$, $N$ are $R$-modules such that $\n(M\otimes_RN)=0$, then $\Tor^R_n(M,\nu_N)=0$ for all $n$ and 
$$
\Po^R_{M,N}(t)=\frac{\Hi_M(-t)\Hi_N(-t)}{\Hi_R(-t)}. 
$$
\end{lemma}
\begin{proof}
Let $\varphi\colon R^a\to N$ be a surjective homomorphism and let $\psi\colon \n R^a\to \n N$ be the map induced by $\varphi$. Note that $\psi$ is a surjective map of $\kk$-vector spaces, and hence it is split, implying that $\Tor_i^R(M,\psi)$ is surjective for all $i$.  We have a commutative diagram, where the vertical arrow on the left is induced by the inclusion $\n R^a\subseteq R^a$:  
\begin{equation*}
	\begin{tikzcd}[column sep=18mm, row sep=5mm]   
		& \Tor_i^R(M,\n R^a) \arrow[d] \arrow[r, twoheadrightarrow, "{\Tor_i^R(M,\psi)}"] & \Tor_{i}^R(M,\n N) \arrow[d, "{\Tor_i^R(M,\nu_N)}"] \\
		& \Tor_i^R(M, R^a) \arrow[r, "{\Tor_i^R(M,\varphi)}"] & \Tor_i^R(M, N)
	\end{tikzcd}
\end{equation*}
Since $\Tor_i^R(M,R^a)=0$ when $i>0$, it follows that $\Tor_i^R(M, \nu_N)=0$ for all $i>0$. The hypothesis that $\n(M\otimes_RN)=0$ implies that 
$\Tor_0^R(M, \nu_N)=0$ as well. 

Given the hypothesis $\n^2=0$, it is clear from the definition that any finitely generated $R$-module is Koszul. Furthermore, for all $i$, the fact that $\Tor_i^R(M, \nu_N)=0$ implies we have a short exact sequence in homology, which is induced from \eqref{sesm}:
$$0 \to \Tor_{i}^R(M,N) \xrightarrow{} \Tor_{i}^R(M,\overline{N}) \xrightarrow{} \Tor_{i-1}^R(M,\n N) \to 0 \,.$$
A rank count in this exact sequence, together with the formula \eqref{koszul:p},  gives
\begin{align*}
\Po_{M,N}^R(t) &= \Po_{M,\overline{N}}^R(t) -t\Po_{M,\n N}^R(t) \\
&= \ell(\overline{N})\Po_{M,{\sf k}}^R(t) - \ell(\n N)t\Po_{M,{\sf k}}^R(t) \\
&= \Po_{M,{\sf k}}^R(t)\cdot \left(\ell(\overline{N}) - \ell(\n N)t\right) \\
&= \dfrac{\Hi_M(-t)\Hi_N(-t)}{\Hi_R(-t)}\,. \qedhere
\end{align*}
\end{proof}

\begin{bfchunk}{Short local rings.}
Let $(Q,\m,\kk) $ be a local Artinian ring. We say that $Q$ is {\it short} if $\m^3 = 0$. While one may perhaps expect relatively simple homological behavior over such rings, this expectation is only partially met; see the introduction of \cite{avramov} for an overview of known bad behavior over such rings. An element $f\in Q$ such that $f^2=0$ and $f\m=\m^2$ is called a {\it Conca generator} in \cite{avramov}.  
Note that rings that admit a Conca generator are necessarily short. In \cite[Theorem 1.1]{avramov} it is proved that if $Q$ has a Conca generator, then $\Po_M^Q(t)$ is rational, with denominator $\Hi_Q(-t)$, for all finitely generated $R$-modules $M$. 

While not all short rings admit a Conca generator, Conca \cite[Theorem 10]{conca2} shows that when $\kk$ is algebraically closed a generic quadratic standard graded algebra $Q$ with $\dim_{\kk} Q_2<\dim_{\kk} Q_1$ admits a Conca generator. 

All Gorenstein short rings with $\kk$ algebraically closed have a Conca generator, see \cite[Theorem 4.1]{avramov}, and in this case 
Menning and \c{S}ega \cite[Theorem 3.1]{melissa} proved that $\Po_{M,N}^Q(t)$ is rational, with denominator $\Hi_Q(-t)$, for all finitely generated $R$-modules $M$, $N$. While this kind of rationality result is rather trivial when $\m^2=0$, as seen in \cref{n2=0}, rationality does not hold for all short rings since there exist short local rings $Q$ for which $\Po_{\kk}^Q(t)$ is irrational, cf.~Anick \cite{anick}. It remains an open question whether all generalized Poincar\'e series are rational in the case of short local rings with a Conca generator. Our work below brings some evidence towards a positive answer. 
\end{bfchunk}

In what follows we  consider short local rings that admit exact zero divisors. We first show in \cref{ez:hil} that existence of exact zero divisors in a short local ring restricts the Hilbert series of the ring. On the other hand, if $Q$ is a short standard graded algebra with $\Hi_Q(t)=1+et+(e-1)t^2$ that has a Conca generator (which is necessarily an exact zero divisor, as seen below), it follows from \cite[Proposition 5.4]{short}  that a generic linear form $f$ is an exact zero divisor. These observations make the point that the hypotheses of our final result, \cref{final}, are satisfied in many situations. 

\begin{lemma}\label{ez:hil}
Let $(Q, \m, \kk)$ is a short local ring, and let $e$ denote the minimal number of generators of $\m$. 

\begin{enumerate}
    \item If there exists an exact pair $(f,g)$ of zero divisors in $Q$, then
$$
\Hi _Q(t)=1+et+(e-1)t^2 \quad\text{ and } \quad f\m=\m^2\, ,\, g\m=\m^2\,,\, f,g\notin \m^2
$$
\item If $\Hi _Q(t)=1+et+(e-1)t^2$, then $(f,f)$ is an exact pair of zero divisors in $Q$ if and only if $f$ is a Conca generator. 
\end{enumerate}
\end{lemma}
\begin{proof}
(1) If $e=1$, then $\m=(f)=(g)$ and hence $\m^2=0$. The statement is thus clear. Assume now $e\ge 2$.
Indeed, observe first that $g\notin \m^2$, since $g\in \m^2$ implies $g\m=0$, and hence $(f)=\ann_Q(g)=\m$. This is a contradiction, since $e\ge 2$. Similarly, $f\notin \m^2$

Since $\m^3=0$, note that $\m(f \m)=0$. We have a short exact sequence of $\kk$-vector spaces 
$$
0\to \frac{(g)+\m^2}{\m^2}\to \frac{\m}{\m^2}\xrightarrow{\cdot f}  f\m\to 0\,.
$$ 
Since $\dim_{{\sf k}}\m/\m^2=e$ and $$\dim_{\kk}\frac{(g)+\m^2}{\m^2}=1\,,$$
a rank count in this exact sequence yields the second equality below:
$$
\ell(f\m)=\dim_{\kk}(f \m)=e-1\,.
$$
Then, a dimension count in the short exact sequence 
$$
0\to f\m \hookrightarrow  (f)\to \frac{(f)}{f\m}\to 0
$$
gives
$$
\ell(f)=\ell(f \m)+\ell((f)/f\m)=(e-1)+1=e\,.
$$
Similarly, we have $\ell(g)=e$. 
The hypothesis that $(f,g)$ is an exact pair implies $Q/(f)\cong (g)$, and hence we have an exact sequence 
$$
0\to (f)\to Q\to (g)\to 0\,.
$$
A  length count in this sequence gives 
$$
\ell(Q)=\ell(f)+\ell(g)=2e\,.
$$
Since $\ell(Q)=1+\dim_{\kk}(\m/\m^2)+\dim_{\kk}(\m^2)$, we conclude $\dim_{\kk}(\m^2)=e-1$. This establishes the desired formula for $\Hi_Q(t)$. Finally, the inclusion $f\m\subseteq \m^2$ and the fact that $\dim_{\kk}(f\m)=e-1=\dim_{\kk}(\m^2)$ implies $f\m=\m^2$. Similarly, $g\m=\m^2$. 

(2) Assume $\Hi _Q(t)=1+et+(e-1)t^2$. If $(f,f)$ is an exact pair of zero divisors, then the desired conclusion follows from (1). For the converse, assume that $f^2=0$ and $f\m=\m^2$. Since $(f)\subseteq \ann_Q(f)$, it suffices to show that $(f)$ and $\ann_Q(f)$ have the same length. 

The hypothesis that $\Hi _Q(t)=1+et+(e-1)t^2$ implies $$\ell(\m)=2e-1,\qquad \ell(f\m)=\ell(\m^2)=e-1\qquad \text{and}\qquad  \ell(Q)=2e\,.
$$
Using a length count in the exact sequence
\[
0\to \ann_Q(f)\hookrightarrow \m\xrightarrow{\cdot f} f\m\to 0
\]
we obtain 
$$
\ell(\ann_Q(f))=\ell(\m)-\ell(f \m)=2e-1-(e-1)=e\,,
$$
and then a length count in the exact sequence 
$$
0 \to \ann_Q(f) \hookrightarrow Q \xrightarrow{\cdot f} (f) \to 0
$$
gives 
\begin{align*}
\ell(f)&=\ell(Q)-\ell(\ann_Q(f))=2e-e=e=\ell(\ann_Q(f))\,.\qedhere
\end{align*}
\end{proof}

 In \cite[Theorem 3.2]{avramov} it is proved that if $Q$ is a short local ring admitting a Conca generator $f$, then any finitely generated $Q$-module $M$ with $fM=0$ is Koszul, and thus we have an equality  $\Po_M^Q(t)= \Hi_M(-t)/\Hi_Q(-t)$. \cref{final} below extends this property to generalized Poincar\'e series, in the case when $Q$ admits exact zero divisors (and thus a Conca generator is also an exact zero divisor, by \cref{ez:hil}).

\begin{theorem} \label{final}
Let $(Q,\m, {\sf k})$ be a short local ring. Assume $\chara({\sf k}) = 0$, $(f,g)$ is an exact pair of zero divisors and set $R=Q/(f)$. Let $M,N$ be finitely generated $R$-modules such that $\m(M \otimes_R N) = 0$ and $gN=0$. 

Then 
$$\Po^Q_{M,N}(t) = \dfrac{\Hi_M(-t)\Hi_N(-t)}{\Hi_Q(-t)} \,.$$
\end{theorem}

\begin{proof}
Let $e$ denote the minimal number of generators of $\m$. By \cref{ez:hil}, we have $f\notin \m^2$, $f\m=\m^2$ and 
$$
\Hi_Q(t)=1+et+(e-1)t^2. 
$$
We have then $\Hi_R(t)=1+(e-1)t$, and hence 
\begin{align*}
 \Hi_Q(t)&=(1+t)\Hi_R(t). 
\end{align*}

Note that  $\m^2N \subseteq fM = 0$. \cref{n2=0} gives that $\Tor_i^R(M,\nu_N) = 0$ for all $i\ge 0$, and then the conclusion follows from \cref{vanish} and \cref{n2=0}. 
\end{proof}


\begin{thebibliography}{99}

\bibitem{anick} D.~J.~Anick, {\it A counterexample to a conjecture of Serre}, Ann. of Math. {\bf 115} (1982); 1-33.

\bibitem{Luchonotes}L.~L.~Avramov, 
\newblock{\em Infinite Free Resolutions}, 
\newblock{Six lectures on commutative algebra}, 
(Bellaterra, 1996), Progr. in Math. {\bf 166}, Birkh\"auser, Basel, (1998); 1-118.

\bibitem{avramov_ci} L.~L.~Avramov, R.~O.~Buchweitz, {\it Support varieties and cohomology over complete intersections}, Invent.math.{\bf 142} (2000), {285--318}.

\bibitem{quasi}  L.~L.~Avramov, I.~Henriques, L.~M.~\c{S}ega, {\it Quasi-complete intersection homomorphisms}, Pure Appl. Math. Q. {\bf 9} (2013), no. 4, {579--612}.

\bibitem{avramov} L.~L.~Avramov, S.~B.~Iyengar, L.~M.~\c{S}ega, {\it Free resolutions over short local rings}, J. London Math. Soc {\bf 78} (2008), {459--476}.

\bibitem{short} L.~L.~Avramov, S.~B.~Iyengar, L.~M.~\c{S}ega, {\it Short Koszul Modules}, Journal of Commutative Algebra {\bf 2} (2010), {249--279}.


\bibitem{hzerodiv} P.~A.~Bergh, O.~Celikbas, D.~A.~Jorgensen, {\it Homological algebra modulo exact zero-divisors}, Kyoto Journal of Mathematics {\bf 54} (2014), {879--895}.

\bibitem{cartan} 
H.~Cartan, S.~Eilenberg
\newblock {\em Homological Algebra},
\newblock{Princeton University Press},
Princeton (1956).




\bibitem{conca2} A.~Conca, 
\newblock{\em Gr\"obner Bases for spaces of quadrics of low codimension}, 
Adv, Appl. Math {\bf 24}, (2000); 111-124. 

\bibitem{conca} A.~Conca, E.~De Negri and M.~E.~Rossi,
\newblock{\em Koszul Algebras and Regularity}, 
\newblock{Commutative Algebra: Expository Papers Dedicated to David Eisenbud on the Occasion of His 65$^{th}$ Birthday}, Springer New York, (2013); 285-315. 

\bibitem{eisenbud}
D.~Eisenbud,
\newblock{\it Homological algebra on a complete intersection, with an application to group representations},
\newblock{Trans.~Amer.~Math.~Soc} {\bf 260} (1980), {35-64}.

 
 \bibitem{gullik} T.~H.~Gulliksen, {\it 
A change of ring theorem with applications to Poincar\'e series and intersection multiplicity}, Math.~Scand. {\bf 34} (1974), {167--183}. 


\bibitem{zerodiv} I.~Henriques, L.~M.~\c{S}ega, {\it 
Free resolutions over short Gorenstein local rings}, Mathematische Zeitschrift {\bf 267} (2009), {645--663}. 

\bibitem{herzog} J.~Herzog, S.~Iyengar, {\it Koszul modules}, J. Pure Appl. Algebra {\bf 201} (2005), {154--188}. 

\bibitem{totref} H.~Holm, {\it 
Construction of totally reflexive modules from an exact pair of zero divisors}, Bull. London Math. Soc. {\bf 43} (2011), no. 2, {278--288}.


\bibitem{ekustin} R.~Kustin, J.~Striuli, A.~Vraciu, {\it 
Exact pairs of homogeneous zero divisors}, Journal of Algebra {\bf 453} (2016), {221--248}.



\bibitem{melissa} M.~C.~Menning, L.~M.~\c{S}ega, {\it Cohomology of finite modules over short gorenstein rings}, Journal of Commutative Algebra {\bf 10} (2018), 63--81.



\bibitem{siree} 
L.~M.~\c{S}ega, D.~Sireeshan,
\newblock {\em dg Module Structures and minimal free resolutions modulo an exact zero divisor},
\newblock{arxiv},
\newblock{\url{https://arxiv.org/abs/2208.04452}}, (2022). 







\bibitem{windle} 
A.~Windle,
\newblock {\em Cohomological operators on quotients by exact zero divisors}, J.~Algebra Appl. {\bf 21} (2022), Paper No. 2250016, 13 pp.


\end{thebibliography}
\end{document}